\documentclass{amsart}

\usepackage{amsmath, amsthm, amssymb}
\usepackage{epsfig}
\usepackage{graphicx}
\usepackage{array}
\usepackage{amsmath}
\usepackage{amsfonts}
\usepackage[all]{xy}
\usepackage{amsfonts}
\usepackage{amsmath}%
\usepackage{amsfonts}%
\usepackage{amssymb}%
\usepackage{graphicx}

\usepackage{amsfonts}
\usepackage{amssymb}
\usepackage{amsmath,amsxtra,amssymb}
\numberwithin{equation}{section}

\newtheorem{theorem}{Theorem}[section]

\newtheorem{remark}[theorem]{Remark}

\newtheorem{proposition}[theorem]{Proposition}

\newcommand{\field}[1]{\mathbb{#1}}
\newcommand{\C}{{\field{C}}}

\newcommand{\id}{{\iota}}
\newcommand{\ot}{{\otimes}}
\newcommand{\om}{{\omega}}
\newcommand{\tp}{{\widehat{\otimes}}}

\newcommand{\N}{{\mathfrak{N}}}
\newcommand{\M}{{\mathfrak{M}}}

\newcommand{\B}{{\mathcal{B}}}

\newcommand{\fee}{{\varphi}}

\newcommand{\LL}{{\mathcal{L}^{\infty}(\G)}}
\newcommand{\LO}{{\mathcal{L}^{1}(\G)}}
\newcommand{\LT}{{\mathcal{L}^{2}(\G)}}
\newcommand{\MG}{{\mathcal{M}(\G)}}

\newcommand{\PG}{{\mathcal{P}(\G)}}

\newcommand{\CZ}{{\mathcal{C}_0(\G)}}
\newcommand{\CU}{{C_u(\G)}}
\newcommand{\hp}{{\mathcal{H}_\mu^p(\G)}}
\newcommand{\hhp}{\tilde{\mathcal{H}}_\mu^p(\G)}
\newcommand{\hpp}{\tilde{\mathcal{H}}_\mu^p(\G)}
\newcommand{\hq}{{\tilde {\mathcal{H}}_\mu^q(\G)}}
\newcommand{\ho}{{\mathcal{H}_\mu^1(\G)}}
\newcommand{\jp}{{\mathcal{J}_\mu^p(\G)}}
\newcommand{\jq}{{\tilde{\mathcal J}_\mu^q(\G)}}
\newcommand{\Ep}{{E_\mu^p}}
\newcommand{\Eq}{{{\tilde E}_\mu^q}}
\newcommand{\EEp}{{{\tilde E}_\mu^p}}
\newcommand{\Lp}{{\mathcal{L}^p(\G)}}
\newcommand{\LLp}{\tilde{{\mathcal{L}}}^p(\G)}
\newcommand{\Lq}{{\tilde{\mathcal{L}}^q(\G)}}
\newcommand{\nf}{{\mathfrak{N}_\fee}}

\newcommand{\G}{\mathbb G}

\def\proclaim #1. #2\par{\medbreak
\noindent{\bf#1.\enspace}{\sl#2}\par\medbreak}
\allowdisplaybreaks

\title[Harmonic NC $L^p$-Opertors on LCQ Groups]
{On Harmonic Noncommutative $L^p$-Operators on Locally Compact Quantum  Groups}
\author{Mehrdad Kalantar}
\address{ School of Mathematics and Statistics,
        Carleton University, Ottawa, Ontario, Canada K1S 5B6}
\email{mkalanta@math.carleton.ca}

\begin{document}
\begin{abstract}
For a locally compact quantum group $\G$ with tracial Haar weight $\fee$, and a quantum measure $\mu$ on $\G$, we study the space $\hp$ of 
$\mu$-harmonic operators in the non-commutative $L^p$-space $\Lp$ associated to the Haar weight $\fee$.
The main result states that if $\mu$ is non-degenerate, then $\hp$ is trivial for all $1\leq p<\infty$. 
\end{abstract}


\maketitle


\section{Introduction and Preliminaries}

Noncommutative Poisson boundaries of (discrete) quantum groups $\G$
was first introduced and studied by Izumi in \cite{I}.
Motivated by the classical setting, in fact, he defined the 
Poisson boundary of $\G$ associated to a `quantum measure' $\mu$,
as the space of $\mu$-harmonic `functions', i.e., the fixed point space of the Markov operator
associated to $\mu$.
For discrete quantum groups, this has been further studied by several authors (cf. \cite{INT}, \cite{VV2}, \cite{VV3}).
Poisson boundaries in the locally compact quantum group setting has been studied
by Neufang, Ruan and the author in \cite{KNR}.
Quantum versions of several important classical results regarding harmonic functions were proved there.
In particular, triviality of special classes of harmonic functions, such as $\mathcal{C}_0$-functions, was proved.

Another important fact regarding classical harmonic functions on locally compact groups, is that
for $1\leq p <\infty$, any $L^p$-harmonic function associated to an adapted probability measure is trivial.
The main result of this paper is a quantum version of this result.
But, in order to talk about $\mu$-harmonic elements in the non-commutative $L^p$-spaces,
we first need to define the convolution action by $\mu$ on such spaces.

In his PhD thesis \cite{Cooney}, Cooney has studied the noncommutative $L^p$-spaces 
associated to the Haar weight $\fee$ of a locally compact quantum group $\G$.
He mainly considered the Haagerup's version, and could prove
that in the Kac algebra setting, the convolution action of an `absolutely continuous quantum measure'
can be extended to the Haagerup noncommutative $L^p$-spaces.
So, we cannot consider harmonic operators in the general setting of all locally compact quantum groups.
Moreover, in the case of non-tracial $\fee$, there are different ways to define the non-commutative
$L^p$-spaces. Although, all these spaces are isometrically isomorphic as Banach spaces, but
the identifications are not necessarily compatible with the quantum group structure,
and so it is not clear whether the space of $\mu$-harmonic $L^p$-operators is the same, as a Banach space,
for all different definitions of non-commutative $L^p$-spaces.

So, in this paper, instead of restricting ourselves to the Kac algebra setting,
we consider locally compact quantum groups $\G$ whose Haar weight $\fee$ is a trace.
In this case, the convolution action is extended to the noncommutative $L^p$-spaces, and
the main result of the paper states that
in the case of a non-degenerate quantum measure $\mu$,
for $1\leq p <\infty$, any $\mu$-harmonic element which lies in the noncommutative $L^p$-space of $\fee$,
is trivial.


First, let us introduce our terminology and recall some results on locally compact quantum groups which we will be using in this paper.
For more details, we refer the reader to \cite{KV1}.

A \emph{locally compact quantum group} $\G$ is a
quadruple $(M, \Gamma, \varphi, \psi)$, where $M$ is a
von Neumann algebra with a co-associative co-multiplication
$\Gamma: M\to M \bar\otimes M$, and $\varphi$
and  $\psi$ are  (normal faithful semi-finite) left and right
Haar weights on $M$, respectively. 
We write $\M_\fee^+ = \{x\in M^+ : \fee(x)<\infty\}$ and
$\N_\fee = \{x\in M^+ : \fee(x^*x)<\infty\}$, and we denote
by $\Lambda_\fee$ the inclusion of $\N_\fee$ into the GNS Hilbert
space $H_\fee$ of $\fee$.
For each locally compact
quantum group $\G$, there exist a  \emph{left fundamental unitary
operator}  $W$ on $H_\fee\otimes H_\fee$
which satisfies  the  pentagonal relation
and such that the co-multiplication $\Gamma$ on $M$ can be expressed as
\begin{equation*}
\Gamma(x) = W^{*}(1\otimes x)W
\quad(x \in M).
\end{equation*}
There exists an anti-automorphism $R$ on $M$, called the {\it unitary antipode}, such that $R^2 = \id$, and
\[
\Gamma\circ R = \chi(R\ot R)\circ\Gamma,
\]
where $\chi(x\ot y) = (y\ot x)$ is the flip map. It can be easily seen that if $\fee$ is a left Haar weight, 
then $\fee R$ defines a right Haar weight on $M$.

Let $M_*$ be the predual of $M$. Then the pre-adjoint of
$\Gamma$ induces on $M_*$ an associative completely
contractive multiplication
\begin{equation*}
 \star  :  M_*\hat\otimes M_*\ni f_{1} \otimes f_{2}
 \,\longmapsto\, f_{1} \star f_{2} = (f_{1} \otimes
f_{2})\circ \Gamma \in M_*.
\end{equation*}

The \emph{left regular representation} $\lambda : M_* \to
\B(H_\fee)$ is defined by
 \[
\lambda : M_*\ni f   \,\longmapsto\, \lambda(f) = (f\otimes \iota)(W)
\in \B(H_\fee),
 \]
which is an injective and completely contractive algebra homomorphism
from $M_*$ into $\B(H_\fee)$.
Then  $\hat M={\{\lambda(f): f\in M_*\}}''$
is the von Neumann algebra associated with the dual quantum group
$\hat \G$.  It follows that $W \in M \bar \otimes \hat M$.
We also define the completely contractive injection
\[
\hat\lambda:  {\hat M}_*\ni\hat f \,\longmapsto\, \hat\lambda(\hat f) =
(\iota \otimes \hat f)(W)\in M.
\]
The \emph{reduced quantum group $C^*$-algebra}  
$$\mathcal{C}_{0}(\G) = \overline{\hat\lambda(L_{1}(\hat \G) )}^{\|\cdot\|}$$
is a weak$^*$ dense $C^*$-subalgebra of $M$.
Let $\MG$ denote the operator dual $\CZ^{*}$.
There exists  a completely contractive multiplication on $\MG$ given by
the convolution
\[
\star : \MG\tp \MG\ni \mu\ot \nu \,\longmapsto\, \mu \star \nu 
= \mu (\id\otimes \nu)\Gamma = \nu (\mu \otimes \id)\Gamma
\in \MG
\]
such that  $\MG$ contains $M_*$ as a norm closed two-sided ideal.
 Therefore, for each $\mu\in\MG$, we obtain a pair of completely bounded maps
\begin{equation*}
f\, \longmapsto \, \mu \star f\ \  ~ \mbox{and } ~ \ \ f \,\longmapsto\, f \star \mu 
\end{equation*}
on $M_*$ through the left and right convolution products of $\MG$.
The adjoint maps give the convolution actions $x\mapsto \mu\star x$ and $x\mapsto x\star\mu$
that are normal completely bounded maps on $M$.


We denote by $\PG$ the set of all states on $\CZ$ (i.e., `the quantum probability measures'). 
For any such element the convolution action is a \emph{Markov operator}, i.e.,
a unital normal completely positive map, on $M$.

Now assume that the left Haar weight $\fee$ on $\G$ is a trace, and let $\psi = \fee R$ be the right Haar weight.
We denote by $\Lp$ and $\tilde{\mathcal{L}}^p(\G)$ the noncommutative $L^p$-spaces associated to $\fee$ and $\psi$, respectively;
these spaces are obtained by taking the closure
of $\M_\fee$ and $\M_\psi$ under the norms $\|x\| = \fee(|x|^p)^{\frac1p}$ and $\|x\| = \psi(|x|^p)^{\frac1p}$, respectively (see \cite{Tak2} for details).
We denote by $\LL$ the von Neumann algebra $M$.
Similar to the classical case, one can also construct the non-commutative ${L}^p$-spaces, using the complex interpolation method 
(cf. \cite{Iz}, \cite{Kosaki}, \cite{Terp}).
The map
\begin{equation}\label{je1}
\M_\fee\ni x\,\longmapsto\, \fee\cdot x \in M_*
\end{equation}
extends to an isometric isomorphism between $\LO$ and $M_*$,
where $\langle \fee\cdot x \,,\, y \rangle = \fee(xy)$.


\section{$\mu$-harmonic operators}
We assume that $\mu\in\PG$ throughout this section.
By invariance of the left Haar weight $\fee$,
we can easily see that $\Lp\cap\LL$ is invariant under the left convolution action by $\mu$.  
Since $\fee = \psi R$ is a trace, by \cite[Proposition 5.20]{KV1}, we have
\[
R\big((\id\ot\fee)\Gamma(a)(1\ot b)\big) = (\id\ot\fee)\big((1\ot a)\Gamma(b)\big)
\]
for all $a,b\in\nf$. 
Therefore we obtain
\begin{eqnarray}
\langle\, \mu\star(\fee \cdot a) \,,\, b \,\rangle &=& \langle\, \mu\,,\, (\id\ot\fee) \big((1\ot a) \Gamma(b)\big) \,\rangle =
\langle\,\mu R \,,\, (\id\ot\fee)\big(\Gamma(a)(1\ot b)\big)\,\rangle \\ 
&=& \left\langle\,(\mu R) \star (b\cdot\fee)\, ,\, a\,\right\rangle 
= \left\langle\, (b\cdot\fee)\, ,\, (\mu R) \star a\,\right\rangle 
=  \left\langle\, \fee\cdot((\mu R) \star a)\, ,\, b\,\right\rangle.
\end{eqnarray}
Since the map (\ref{je1}) is an isometry, this shows that the convolution action
\[
\M_\fee \ni x \,\longmapsto \fee\cdot x \,\longmapsto (\mu R)\star(\fee\cdot x) = \fee\cdot(\mu\star x)\,\longmapsto \mu \star x \in\M_\fee
\]
extends to an operator on $\LO$ with the same norm as the convolution operator by $\mu$ on $\LL$.
Now, interpolating between $\LO$ and $\LL$, we can extend the convolution action
\[
\Lp\cap\LL\ni x\,\longmapsto\, \mu\star x\in\Lp\cap\LL
\]
to $\Lp$.
An operator $x\in\Lp$ is called $\mu$-harmonic if $\mu\star x = x$, and
$\hp = \{x\in \Lp: \mu\star x = x\}$ is the space of $\mu$-harmonic operators.
It is easy to see that $\hp$ is a weak*
closed subspace of $\Lp$ for all $1<p\leq\infty$.

Similar to the case $p=\infty$, we have a projection $\Ep : \Lp\to\hp$ constructed as follows.
Let $\mathcal{U}$ be a free ultra-filter on $\mathbb{N}$, and define
$\Ep: \Lp \to \Lp$ by 
the weak$^*$ limit
\begin{equation*}
\Ep (x) = \lim_{\mathcal{U}} \frac{1}{n}\sum_{k=1}^{n} \mu^k\star x. 
\end{equation*}
Then it is easy to see that $\Ep\circ\Ep = \Ep$, and that $\hp = \Ep(\Lp)$.
Moreover, by considering the convolution action on $\Lp\cap\LL$, and passing to limits,
we can see that $\Ep$ is also positive.

Similarly, we can extend the right convolution action
\[
\LLp\cap\LL\ni x\,\longmapsto\, x\star\mu\in\LLp\cap\LL
\]
to $\LLp$. Then $\hhp = \{x\in \LLp: x\star\mu = x\}$
is a weak* closed subspace of $\LLp$,
and there is a positive projection $\EEp$ on $\LLp$ such that $\hpp = \EEp(\LLp)$.

\begin{proposition}
The unitary antipode $R$ extends to an isometric isomorphism
\[
R \,:\, \Lp\, \to\, \LLp
\]
such that $R(\hp) = \tilde{\mathcal{H}}_{\mu R}^p$ for all $1\leq p \leq \infty$.
\end{proposition}
\begin{proof}
Since $R$ is an anti-automorphism, we have
\[
\psi(|R(a)|^p) = \psi(R(|a|^p)) = \fee(|a|^p) 
\]
for all $a\in\M_\fee$. Therefore $R$ extends to an isometry from $\Lp$ onto $\LLp$.
Moreover, we have
\begin{eqnarray*}
R(\mu\star a) &=& R\big((\mu\ot\id)\Gamma(a)\big) = R\big((\mu\ot\id)\Gamma(R^2(a))\big) \\
&=& R\big((\id\ot\mu)(R\ot R)\Gamma(R(a))\big) \\ &=&
(\id\ot\mu R)\Gamma(R(a)) = R(a)\star\mu R
\end{eqnarray*}
which implies that
$R(\hp) = \tilde{\mathcal{H}}_{\mu R}^p\,$.
\end{proof}
Therefor, for $1<p,q<\infty$ with $\frac1p + \frac1q =1$, we can identify
each of $\Lp$ and $\Lq$ with the dual space of the other, via
\[
\langle a \,,\, b\rangle = \fee(a R(b)) = \psi(R(a)b) \ \ \ \ a\in\Lp\,, \ \ b\in\Lq\,.
\]

\begin{theorem}\label{j2}
Let $1 < p\, , q < \infty$ such that $\frac1p + \frac 1q = 1$, then we have linear isometric isomorphisms
\[
\hp^* \cong \hq\ \ \ \ \  \ \text{and} \ \ \ \ \  \ \hp \cong \hq^*
\]
\end{theorem}
\begin{proof}
Denote
$$\jp := \{x - \mu\star x\,:\, x\in\Lp\}^- \ \ \ \text{and} \ \ \
\tilde{\mathcal J}_\mu^q(\G) := \{y - y\star\mu\,:\, y\in\Lq\}^-$$
%
Since
\begin{eqnarray*}
\langle x \,,\, y\star\mu\rangle &=& \psi\big(R(x)(y\star\mu)\big) 
= \psi\big(R(x)(\id\ot\mu)\Gamma(y)\big) 
= \mu\big((\psi\ot\id)(R(x)\ot 1)\Gamma(y)\big) \\
&=& \mu R\big((\psi\ot\id)\Gamma(R(x))(y\ot 1)\big) = 
\psi\big((\id\ot\mu R)\Gamma(R(x))\,y\big) \\
&=& \psi\big(R\big((\mu\ot\id)\Gamma(x)\big)\,y\big) 
= \psi(R(\mu\star x)\,y) = \langle \mu\star x \,,\, y\rangle
\end{eqnarray*}
for all $x\in\M_\fee$ and $y\in\M_\psi$,
it follows that $\hp = \jq^\bot$, and therefore
\[
\hp^* = \frac{\Lq}{\hp^\bot} = \frac{\Lq}{\jq}\,.
\]
In the following we show that the correspondence
\[
\frac{\Lq}{\jq}\ni y + {\jq} \,\longmapsto\, \Eq(y)\in \hq. 
\]
defines a linear isometric isomorphism. First we observe that
\[
\tilde{E}_\mu^q(y\star\mu - y) = \lim_{\mathcal U} \left( (y\star\mu-y)\star\sum_1^n\frac{\mu^k}{n}\right) = 0
\]
for all $y\in\Lq$, which implies that the above map is well-defined. It is obviously onto.
To check the injectivity, first note that
\[
y - y\star\mu^k = (y - y\star\mu) + (y\star\mu - y\star\mu^2) + (y\star\mu^{k-1} - y\star\mu^k)\in\jq,
\]
for all $k\in\mathbb N$. Now suppose that $\Eq(y) = 0$. Then, by above, and weak* closeness of $\jq$, we have
\[
y = y - \Eq(y) = y - \left(\lim_{\mathcal{U}} \frac{1}{n}\sum_{k=1}^{n} y\star\mu^k\right) = \lim_{\mathcal{U}} \frac{1}{n}\sum_{k=1}^{n} \left(y - y\star\mu^k\right)\in\jq.
\]
and therefore the injectivity of the map follows.
Moreover, since $\Eq$ is an idempotent, it follows that
\[
y + \jq = \Eq(y) + \jq.
\]
Therefore
\[
\| y + \jq \| \leq \| \Eq(y)\|.
\]
On the other hand, we have
\begin{eqnarray*}
\| \Eq(y) \| &=& \sup \big\{\big|\,\langle\Eq(y)\,,\,x\rangle\,\big|\,:\, x\in\Lp\,,\ \|x\|\leq1\big\} \\
&=& \sup \big\{\left|\,\langle y \,,\, \Ep(x)\rangle \,\right|\,:\, x\in\Lp\,,\ \|x\|\leq1\big\} \\ &\leq& \|y + \jq\|\,. 
\end{eqnarray*}
This shows that the map is isometric, and so yields the first identification. The second identification is proved along similar lines.
\end{proof}

\begin{proposition}\label{j1}
For $1<p\leq\infty$ the space $\hp$ is generated by its positive elements.
\end{proposition}
\begin{proof}
By considering the polar decomposition, we observe that $\Lp\cap\LL$ is self-adjoint.
Let $x\in\hp$, and $\Lp\cap\LL\ni x_n\rightarrow x$ in $\Lp$. Using the continuity of the adjoint on $\Lp$, we obtain
\[
\mu\star x^* = \lim_n \mu\star x_n^* =  \lim_n (\mu\star x_n)^* = (\lim_n \mu\star x_n)^* = x^*,
\]
where the limits are taken in $\Lp$. Therefore, $\hp$ is self adjoint, and so is generated by its self adjoint elements.
Now, let $x$ be a self-adjoint element in $\Lp$, and let $x = x_{+} - x_{-}$ where both $x_{+}$ and $x_{-}$
are in $\Lp^+$. Then we have
\[
x = \Ep(x) = \Ep(x_{+}) - \Ep(x_{-}),
\]
which yields the result by positivity of the map $\Ep$.
\end{proof}

\subsection*{Main theorem: case $1<p<\infty$}

{\ }

\noindent
A state $\mu\in \PG$ is called  \emph{non-degenerate} on $\CZ$
if for every non-zero element $x\in \CZ^+$,
there exists $n\in\mathbb{N}$ such that
$\langle x \,,\, \mu^n\rangle \neq 0$.

\begin{theorem}\label{j4}
Let $\G$ be a non-compact, locally compact quantum group with a tracial (left) Haar weight $\fee$,
and let $\mu\in\PG$ be non-degenerate. Then for all $1 < p < \infty$ we have
\[
\hp = \{0\}\,.
\]
\end{theorem}
\begin{proof}
First let $1<p\leq 2$, and suppose $0\leq x\in\hp$ with $\|x\|_p = 1$. Define
$$\tilde\mu := \sum_{i=n}^\infty \frac{\mu^n}{2^n}.$$
Since $\mu$ is non-degenerate $\tilde\mu$ is faithful, and $\tilde\mu\star x = x$.
Now, let $q\geq 2$ be such that ${\frac1 p} + {\frac1 q} = 1$. Using the
duality between $\Lp$ and $\Lq$, we assign to each pair $a\in\Lp$ and $b\in\Lq$, an element $\Omega_{a,b}\in\LL$ defined by
\[
\langle f\,,\,\Omega_{a,b}\rangle = \langle f\star a \,,\, b\rangle \ \ \ \ \ \ (f\in M_*).
\]
We clearly have $\|\Omega_{a,b}\| \leq \|a\|_p \|b\|_q$.
Now, choose $y\in\Lq$, $\|y\|_q = 1$, and such that $\langle x\, ,\, y\rangle =1$.
We claim that $\Omega_{x,y}\in\CZ$ (in fact $\Omega_{a,b}\in\CZ$ for all $a\in\Lp$ and $b\in\Lq$). To see this,
assume that
\[
x = \int_0^\infty \lambda\, de_\lambda
\]
is the spectral decomposition of $x$, and let
\[
x_n = \int_{\frac1n}^n \lambda\, de_\lambda.
\]
Then $x_n\in\Lp\cap\LL\subseteq\nf$, $\|x_n\|_p \leq \|x\|_p$, and $\|x-x_n\|_p \to 0$.
Also let $y_n\in\N_\psi$ be such that
\[
\|y_n - y\|_q \rightarrow 0.
\]

Denote by $\om_{\eta,\zeta}$ the vector functional associated with $\eta\,,\zeta\,\in H_\fee$.
Then, for $f\in M_*$ we have
\begin{eqnarray*}
\langle f \,,\, \Omega_{x_n,y_n}\rangle &=& \langle f\star x_n \,,\, y_n\rangle = 
\langle \lambda(f)\Lambda_\fee(x_n)\, ,\, \Lambda_\fee(R(y_n))\rangle \\
&=& \langle \omega_{\Lambda_\fee(x_n)\,,\,\Lambda_\fee(R(y_n))} , \lambda(f)\rangle
\langle f\, ,\, \hat\lambda(\omega_{\Lambda_\fee(x_n),\Lambda_\fee(R(y_n))})\rangle,
\end{eqnarray*}
which implies that $\Omega_{x_n,y_n} = \hat\lambda(\omega_{\Lambda_\fee(x_n),\Lambda_\fee(R(y_n))})\in\CZ$.
Moreover, it follows that
\begin{eqnarray*}
\|\Omega_{x,y} - \Omega_{x_n,y_n}\|_\infty &\leq& \|\Omega_{x - x_n , y}\|_\infty + \|\Omega_{x_n , y - y_n}\|_\infty \\
&\leq& \|x - x_n\|_p\,\|y\|_q + \|x_n\|_p \, \|y - y_n\|_q\,
\rightarrow\, 0.
\end{eqnarray*}
This shows that $\Omega_{x , y}\in\CZ$, as claimed.
But then we have $\|\Omega_{x,y}\| \leq \|x\|_p \|y\|_q = 1$, and
\[
\langle \tilde\mu\,,\,\Omega_{x,y}\rangle = \langle \tilde\mu\star x\, ,\, y\rangle = \langle x\,,\,y\rangle = 1.
\]
Since $\tilde\mu$ is faithful, it follows that $\Omega_{x,y} =1$,
and therefore $1\in\CZ$, which contradicts our assumption of $\G$ being non-compact, so $x=0$.
This shows, by Proposition \ref{j1}, that $\hp = \{0\}$ for all $1<p\leq2$.
Now, a similar argument yields $\hq = 0$ for all $1< q \leq 2$, which implies
by Theorem \ref{j2} that $\hp = \hq^* = \{0\}$ for all $2\leq p<\infty$.
\end{proof}

\subsection*{Main theorem: case $p = 1$}

{\ }

\noindent
Since $\LO$ is not a dual Banach space, our proof for $1< p<\infty$ does not work in this case, and so
we have to treat this case separately. We do this by first proving a similar result for $M_*$
and then using the identification of the latter with $\LO$.
Note that for the following theorem we don't assume that the Haar weight is a trace.
\begin{theorem}\label{j3}
Let $\G$ be a non-compact, locally compact quantum group,
and let $\mu\in\PG$ be non-degenerate. If $\om\in\MG$ is such that $\mu\star\om = \om$, then $\om = 0$.
\end{theorem}
\begin{proof}
Assume that $\mu\star\om = \om$, and let $\tilde\mu$ be as in the proof of Theorem \ref{j4}.
So, $\tilde\mu$ is faithful, and $\tilde\mu\star\om = \om$. Therefore we have
\[
\lambda(\tilde\mu)\lambda(\om)\xi = \lambda(\tilde\mu\star\om)\xi = \lambda(\om)\xi
\]
for all $\xi\in H_\fee$. Now if $\om\neq 0$, there exist $\xi\in H_\fee$ such that $\|\lambda(\om)\xi\| = 1$.
Denote by $\hat\om$ the restriction of $\omega_{\lambda(\om)\xi}$ to $\hat M$. Then $\|\hat\omega\|  = 1$, and
\[
\langle \tilde\mu\,,\,\hat\lambda(\hat\om)\rangle = \langle \lambda(\tilde\mu)\,,\, \hat\om\rangle
= \langle \lambda(\tilde\mu)\lambda(\om)\xi\, ,\, \lambda(\om)\xi\rangle
= \langle \lambda(\om)\xi\, , \,\lambda(\om)\xi\rangle
= 1.
\]
Since $\|\hat\lambda(\hat\om)\| \leq 1$, and $\tilde\mu$ is faithful, it follows that $\hat\lambda(\hat\om) = 1$.
But this implies that $1\in\CZ$, which contradicts our assumption of $\G$ being non-compact.
Hence, $\om = 0$.
\end{proof}

\begin{theorem}\label{j13}
Let $\G$ be a non-compact, locally compact quantum group with a tracial (left) Haar weight $\fee$,
and let $\mu\in\PG$ be non-degenerate. Then
\[
\ho = \{0\}\,.
\]
\end{theorem}
\begin{proof}
Let $x\in\ho$. We have $\mu R\in\PG$ is non-degenerate, and from equations $(2.1)$ and $(2.2)$ we get
\[
\mu R\star (\fee\cdot x) = \fee\cdot(\mu \star x) = \fee\cdot x.
\]
Hence, $\fee\cdot x = 0$ by Theorem \ref{j3}, and therefore $x=0$.
\end{proof}

\begin{remark}
The statements of Theorems \ref{j4} and \ref{j13} are not true in general for the case $p = \infty$.
Any non-degenerate probability measure on a non-amenable discrete group is a counter-example \cite{1}.
\end{remark}

\subsection*{Compact case.}

{\ }

\noindent
We conclude by proving the triviality of $\mu$-harmonic operators in the compact quantum group setting.
\begin{theorem}
Let $\G$ be a compact quantum group with tracial Haar state, and $\mu\in\PG$ be non-degenerate.
Then $\hp = \C1$ for all $1 \leq p \leq \infty$.
\end{theorem}
\begin{proof}
The case $p=\infty$ was proved in the general case in \cite{KNR}. Let $1\leq p<\infty$, and assume that $x\in\hp$, $x\notin\C1$, and $\|x\|_p = 1$. 
Then there exists $y\in\Lq$ with $\|y\|_q = 1$ such that $\langle x \,,\, y\rangle = 1$ (we let $q=\infty$ for $p=1$) and
$\langle 1\, ,\, y \rangle = 0$.
Then from the proof of Theorem \ref{j4} (which we can also apply to the case of $p=1$ and $q=\infty$,
since $\LL\subseteq\LT$ for a compact quantum group) we have $\Omega_{x,y} =1$, and
\begin{equation}\label{j8}
\langle \fee\star x \,,\, y\rangle = \langle \fee \,,\, 1\rangle = 1,
\end{equation}
where $\fee$ is the Haar state on $\G$.
Now, let $x_n\in \LL$ such that $\|x_n - x\|_p \rightarrow 0$. Then
\[
\langle \fee\star x\, ,\, y\rangle = \lim_n \langle \fee\star x_n \,,\,y\rangle = \lim_n \langle \fee\, ,\, x_n \rangle  \langle 1\, ,\,y\rangle = 0.
\]
But this contradicts (\ref{j8}), and therefore $x=0$. Hence, $\hp = \C1$. 
\end{proof}

\begin{remark}
All our results in this paper can be proved, by slight modifications of the arguments, for a state $\mu$ on
the universal $C^*$-algebra $\CU$.
\end{remark}

\bibliographystyle{plain}

\begin{thebibliography}{99}


















\bibitem {BT} E. B\'{e}dos \and L. Tuset,
\textit{Amenability and co-amenability for locally compact quantum groups}, Int.
J. Math. \textbf{14} (2003), 865--884.





\bibitem{CD} G. Choquet \and J. Deny,
\textit{Sur l\'{e}quation de convolution $\mu =\mu \ast \sigma$}, C. R. Acad. Sci. Paris
 \textbf{250} (1960), 799--801.




\bibitem{ChuLp} C.-H. Chu,
\textit{Harmonic function spaces on groups}, J. London Math. Soc. (2)
\textbf{70} (2004), no. 1, 182--198.




\bibitem {Cooney} T. J. Cooney,
\textit{Noncommutative $L_p$-Spaces Associated With locally compact quantum groups},
Ph.D. thesis, University of Illinois at Urbana--Champaign, 2010.












































\bibitem{Iz} H. Izumi, 
{\it Constructions of non-commutative Lp-spaces with a complex
parameter arising from modular actions}, Internat. J. Math. {\bf 8}, No. 8
(1997), 1029--1066.




\bibitem {I} 
M. Izumi, \textit{Non-commutative Poisson boundaries and compact quantum group actions}, 
Adv. Math. \textbf{169} (2002), no. 1, 1--57. 

	






\bibitem {INT}
M. Izumi, S. Neshveyev \and L. Tuset,
 \textit{Poisson boundary of the dual of ${SU}_q(n)$}, Comm. Math. Phys.
\textbf{262} (2006), no. 2,  505--531.







\bibitem {1} V. A. Kaimanovich \and A. M. Vershik,
{\it Random walks on discrete groups: boundary and entropy}, Ann.  Probability,
\textbf{11} (1983), 457--490.



\bibitem {KNR} M. Kalantar, M. Neufang \and Z.-J. Ruan,
\textit{Poisson boundaries over locally compact quantum
groups}, preprint (arXiv: 1111.5828).




\bibitem{Kosaki} H. Kosaki, 
{\it Applications of the Complex Interpolation Method to a von
Neumann Algebra: Non-commutative Lp-spaces}, J. Funct. Anal. {\bf 56}
(1984), 29--78.








\bibitem {KV1} J. Kustermans \and S. Vaes, \textit{Locally compact quantum groups},
Ann. Sci. Ecole Norm. Sup. \textbf{33} (2000), 837--934.






































\bibitem {Tak2} M. Takesaki,
{\it Theory of Operator Algebras. II.} Encyclopaedia of Mathematical Sciences 125. 
Operator Algebras and Non-commutative Geometry, 6. Springer-Verlag, Berlin (2003).

\bibitem {Terp} M. Terp, 
{\it Interpolation between a von Neumann algebra and its predual},
J. Operator Theory {\bf 8} (1982), 327--360.






\bibitem {VV2} S. Vaes \and N. Vander Vennet,
\textit{Poisson boundary of the discrete quantum group $\widehat{A_u(F)}$}, Compos. Math.
 \textbf{146} (2010), no. 4,  1073--1095.



\bibitem {VV3} S. Vaes \and R. Vergnioux,
\textit{The boundary of universal discrete quantum groups, exactness, and factoriality}, Duke Math. J.
 \textbf{140} (2007), no. 1,  35--84.


















\end{thebibliography}

\end{document}